\documentclass[10pt]{amsart}
\usepackage[margin=1.3in]{geometry}
\newtheorem{theorem}{Theorem}[section]
\newtheorem{lemma}[theorem]{Lemma}
\newtheorem{proposition}{Proposition}[section]

\theoremstyle{definition}

\theoremstyle{remark}
\newtheorem{remark}[theorem]{Remark}

\numberwithin{equation}{section}

\usepackage{graphicx, subfigure, caption2}
\usepackage{mathtools}
\begin{document}
\title[Entropy production estimate for ES-BGK model]{Entropy production estimate for the ES-BGK model with the correct Prandtl number}
\author{Doheon Kim}
\address{School of Mathematics, Korea Institute for Advanced Study, Seoul 02455, Korea (Republic of)}
\email{doheonkim@kias.re.kr}
\author{Myeong-Su Lee }
\address{Department of Mathematics, Sungkyunkwan University, Suwon 440-746, Korea (Republic of)}
\email{msl3573@skku.edu}
\author{SEOK-BAE YUN }
\address{Department of Mathematics, Sungkyunkwan University, Suwon 440-746, Korea (Republic of)}
\email{sbyun01@skku.edu}

\begin{abstract}
In this paper, we establish the entropy-entropy production estimate for the ES-BGK model,
a generalized version of the BGK model of the Boltzmann equation introduced for better approximation in the Navier-Stokes limit.
Our result improves the previous entropy production estimate \cite{Yun Entropy} in that (1) the full range of Prandtl parameters $-1/2\leq\nu <1$ including the critical case $\nu=-1/2$ is covered, and (2)  a sharper entropy production bound is obtained. An explicit characterization of the coefficient of the entropy-entropy production estimate is also presented.  
\end{abstract}
\maketitle
\section{Introduction}
The BGK model \cite{BGK} has been widely used as a model equation of the Boltzmann equation in various practical flow problems since the BGK model produces quantitatively reliable results at a much lower computational cost. 
The Navier-Stokes limit of the BGK model, however, shows
a slight inconsistency with the one derived from the Boltzmann equation or experimental data 
in that the Prandtl number - the ratio between the viscosity and the heat conductivity - computed using the BGK model is not correct.   
Halway \cite{Holway} introduced the ellipsoidal BGK model to overcome this drawback by introducing a parameter $\nu$ and generalizing the local Maxwellian in the original BGK model into an ellipsoidal Gaussian parametrized by $\nu$:  
\begin{align}\label{pde}
\begin{split}
&\partial_tf + v\cdot \nabla_xf = A_\nu(\mathcal{M}_\nu(f)-f),\\
&\hspace{1.2cm}f(x,v,0)=f_0(x,v).
\end{split}
\end{align}
The non-negative function $f(x,v,t)$ is the velocity distribution function representing the mass density  on the position $x\in \Omega\subseteq\mathbb{R}^3$ with the velocity $v\in \mathbb{R}^3$ at time $t\geq0$. $A_\nu$ is the collision frequency defined by $A_\nu=\rho^\alpha T^\beta/(1-\nu)$ for some $\alpha,\beta\geq0$. Throughout this paper, we fix the spatial domain to be $\Omega =$ $\mathbb{R}^3$ or $\mathbb{T}^3$ to avoid boundary issues.

The ellipsoidal Gaussian (or non-isotropic Gaussian) is defined by
\begin{align*}
\mathcal{M}_\nu(f):=\frac{\rho }{\sqrt{\mathrm{det}(2\pi\mathcal{T}_{\nu})}}\exp{\left(-\frac{1}{2}(v-U)^{\top}\mathcal{T}^{-1}_{\nu}(v-U)\right)},
\end{align*}
where the local density $\rho$, the local bulk velocity $U$ and the local temperature $T$ are defined by
\begin{align*}
\rho(t,x)&=\int_{\mathbb{R}^3} f(x,v,t)dv\\
U(t,x)&=\frac{1}{\rho}\int_{\mathbb{R}^3}vf(x,v,t)dv\\
T(t,x)&=\frac{1}{3\rho}\int_{\mathbb{R}^3}|v-U|^2f(x,v,t)dv,
\end{align*}
and the temperature tensor $\mathcal{T}_\nu$ is given by
\begin{align*}
\mathcal{T}_\nu=(1-\nu)TI+\nu\Theta.\qquad (-1/2\leq\nu<1)
\end{align*}
Here, $I$ denotes the 3-by-3 identity matrix, and $\Theta$ denotes the stress tensor:
\begin{align*}
\Theta(t,x)&=\frac{1}{\rho}\int_{\mathbb{R}^3}f(x,v,t)(v-U)\otimes(v-U)dv.
\end{align*}
For later use, we observe that
\begin{align*}
&\hspace{1cm}\mathcal{M}_0(f)=\frac{\rho}{\sqrt{(2\pi T)^3}}e^{-\frac{|v-U|^2}{2T}},\cr
&\mathcal{M}_1(f)=\frac{\rho }{\sqrt{\mathrm{det}(2\pi\Theta)}}\exp{\left(-\frac{1}{2}(v-U)^{\top}\Theta^{-1}(v-U)\right)}.
\end{align*}
Under the assumption of $\mathcal{T}_{\nu}>0$, it can be shown by the change of variable 
$X=\mathcal{T}^{-1/2}_{\nu}(v-U)$ that the following cancellation holds: 
\[
\int_{\mathbb{R}^3}\left\{\mathcal{M}_{\nu}(f)-f \right\}(1,v,|v|^2)dv=0\qquad (-1/2\leq\nu<1),
\]
which leads to the conservation of mass, momentum and energy. It was shown in \cite{ALPP} that the
relaxation operator also satisfies 
\[
\int_{\mathbb{R}^3}\left\{\mathcal{M}_{\nu}(f)-f \right\}\ln fdv\leq 0 
\]
from which the celebrated H-theorem can be derived:
\[
\frac{d}{dt}\int_{\mathbb{R}^3}f\ln f dxdv\leq 0.
\] 
\subsection{Main result: Cercignani type entropy production estimate}
Cercignani conjectured in \cite{Cercignani} that the entropy production term of the Boltzmann equation is controlled by the relative entropy between the distribution function and the corresponding Maxwellian through the so-called entropy-entropy production estimate, which would lead to the exponentially fast
convergence of the distribution function to the global Maxwellian, at least in the homogeneous case. Although the conjecture was disproved in \cite{BC,W}, Villani showed that a nonlinear variant of it can be secured \cite{Villani}.
In the case of the original BGK model, such Cercignani type entropy production estimate can be derived directly from the convexity property of the $H$-functional. The same convexity argument, however, does not lead to the desired result for the ES-BGK model 
since the entropy production functional of the ES-BGK model is bounded only by the relative entropy between the distribution function and the ellipsoidal Gaussian. 
This discrepancy was resolved in the case $-1/2<\nu<1$ in \cite{Yun Entropy}, leaving the critical case $\nu=-1/2$ unanswered. In this paper, we extend the result to cover the critical case $\nu=-1/2$, at which the correct Prandtl number is obtained. 

To state our main result, we define the H-functional $H(f)$, the relative entropy $H(f|g)$, and the entropy production functional $D_\nu(f)$:
\begin{align}\label{funcitonals}
\begin{split}
&H(f)=\int_{\mathbb{R}^3}f\ln fdv,\qquad H(f|g)=\int_{\mathbb{R}^3}f\ln(f/g)dv,\\
&\hspace{1cm} D_\nu(f)=\int_{\mathbb{R}^3}A_\nu(\mathcal{M}_\nu-f)\ln fdv.
\end{split}
\end{align}
Throughout this paper, we define $C_{\nu}$ for $-1/2\leq \nu<1$ by
\begin{align}\label{cnux}
C_\nu=\sup_{x>0 }
\frac{\displaystyle3\ln \left(1+\frac{1}{3}x\right)- \ln \left(1+\frac{1+2\nu}{3}x \right)-2\ln \left(1+\frac{1-\nu}{3}x \right)}{\displaystyle3\ln \left(1+\frac{1}{3}x\right)-  \ln(1+x)}.
\end{align}
Then our main result is as follows:
\begin{theorem}\label{main result}
Let $-1/2\leq \nu<1$. Assume that the solution $f$ to (\ref{pde}) is regular enough so that the functionals in (\ref{funcitonals}) are well-defined, and satisfies $\rho(x,t)>0$ for a.e. $x,t$. 	Then, we have
\begin{enumerate}
\item $C_{\nu}$ satisfies the following bound:
\[
C_{\nu}\leq\frac{1}{3}\,\nu^2(5-2\nu)
\]	
on $-1/2\leq \nu<1$. 
\item For such choice of $C_{\nu}$, Cercignani-type entropy-entropy production estimate holds:
	\begin{align*}
	D_\nu(f)\leq-(1-C_\nu)A_\nu H(f|\mathcal{M}_0)
	\end{align*}
	on $-1/2\leq \nu<1$.
	\end{enumerate}
\end{theorem}
\begin{remark}
	(1) We note that $\nu^2(5-2\nu)/3<1$ in the given range of $\nu$. (2)
For example, $f(1+|v|^2+|\ln f|)\in L^1([0,\infty)_t\times\Omega_x\times\mathbb{R}^3_v)$ and $f\geq ae^{-b|v|^2}$ for some positive constants $a$ and $b$, is sufficient for the assumption of the theorem to be satisfied.
\end{remark}
Theorem \ref{main result} improves the previous results on the entropy-entropy production estimate obtained in \cite{Yun Entropy} in the following senses: First, it covers the whole range of the Prandtl parameters: $-1/2\leq \nu<1$ including the critical case $-1/2$, while the estimate in \cite{Yun Entropy} is valid only in $-1/2<\nu<1$.
Secondly, the explicit bound $\nu^2(5-2\nu)/3$ on $C_{\nu}$ is less than the corresponding coefficient  $\max\{2\nu,-\nu\}$ obtained in \cite{Yun Entropy}, with the equality holding only at $\nu=0$ (See Figure 1): 
\[
\frac{1}{3}\nu^2(5-2\nu)\leq \max\{-2\nu,\nu\},
\]
which leads to a sharper entropy-entropy production estimate.
We also note in the figure 1 below that there is a strict gap between $C_{\nu}$ and  $\max\{-2\nu,\nu\}$ at $\nu=-1/2$.
With the choice of $\max\{2\nu,-\nu\}$ in place of $C_{\nu}$, the entropy production estimate given in Theorem \ref{main result} (2) become degenerate at the critical case $\nu=-1/2$ since $1-\max\{2\nu,-\nu\}=\min\{1+2\nu,1-\nu \}=0$ at $\nu=-1/2$.  
\begin{center}
	\includegraphics[height=70mm, width=95mm]{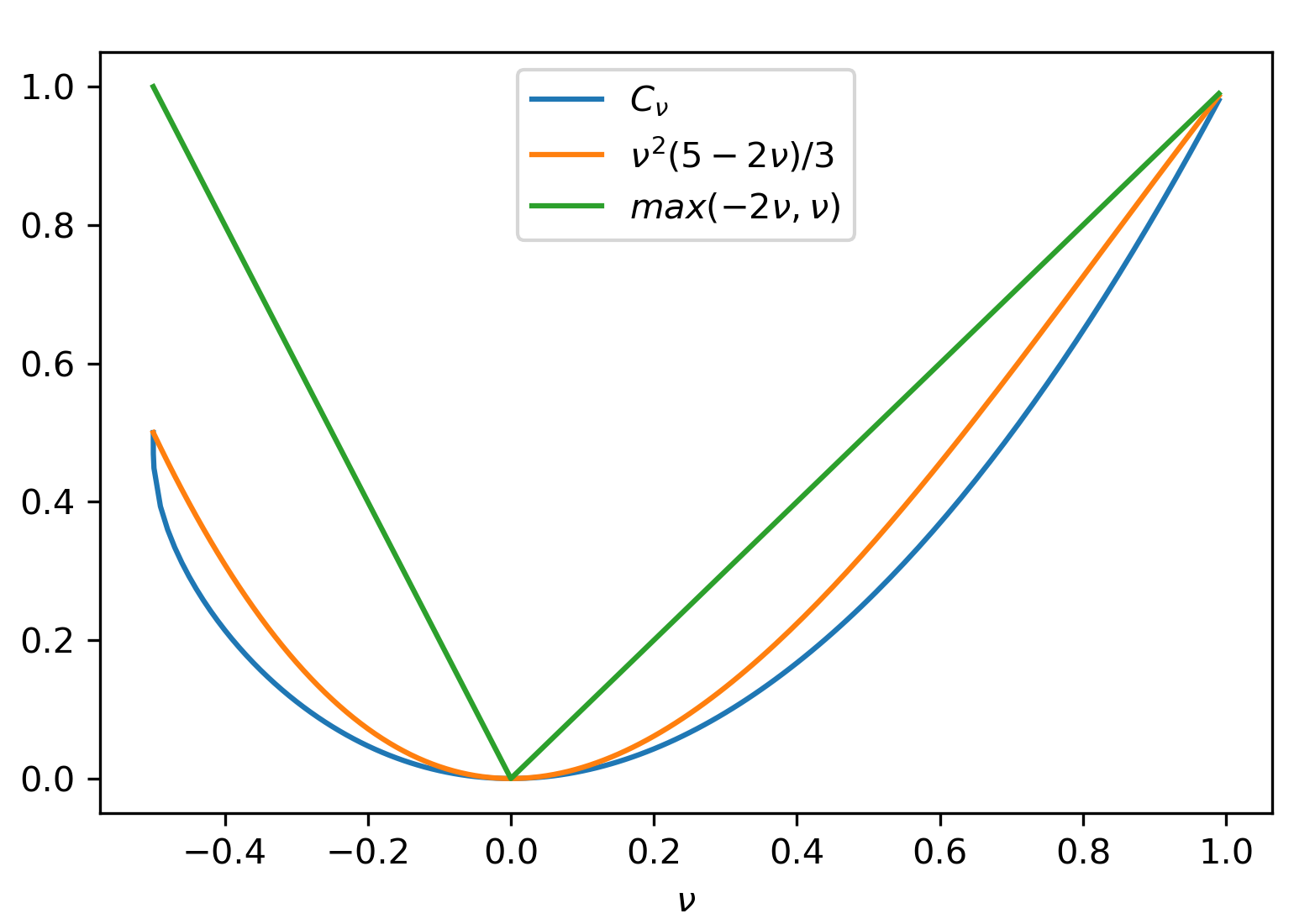}
	\captionof{figure}{Comparison of coefficients}
\end{center}

The starting point of the proof is the observation that the establishment of the following inequality with $0\leq C_{\nu}<1$ is enough to derive the desired result (Lemma \ref{main lemma}):
	\begin{align}\label{we re0}
H(\mathcal{M}_\nu)-H(\mathcal{M}_0)\leq C_\nu\{H(\mathcal{M}_1)-H(\mathcal{M}_0)\}.
\end{align}
We prove this inequality by showing that the range of the following quantity
\begin{align}\label{we re}
\frac{H(\mathcal{M}_\nu)-H(\mathcal{M}_0)}{H(\mathcal{M}_1)-H(\mathcal{M}_0)}
\end{align}
lies in the  interval $\big[0,C_{\nu}\big)$ under reasonable assumptions.
For this, we first rewrite \eqref{we re} as
\begin{align}\label{we re2}
F_{\nu}(\theta_1,\theta_2,\theta_3):=\ln\left(\frac{\prod_{i=1}^{3}\{(1-\nu)\frac{\theta_1+\theta_2+\theta_3}{3}+\nu\theta_i\}}{(\frac{\theta_1+\theta_2+\theta_3}{3})^3}\right)\bigg/  \ln\left(\frac{\theta_1\theta_2\theta_3}{(\frac{\theta_1+\theta_2+\theta_3}{3})^3}\right),
\end{align}
using the eigenfunctions $\theta_i$ $(i=1,2,3)$ of the stress tensor $\Theta$, and  formulate the problem of proving the inequality (\ref{we re0}) into the problem of solving the maximization problem of this three-variable function:
\begin{equation*}
\sup_{\substack{\theta_1,\theta_2,\theta_3>0\\ \exists~i,j:~\theta_i\neq\theta_j}}F_{\nu}(\theta_1,\theta_2,\theta_3).
\end{equation*}
 We then 
reduce this three-variable optimization problem into a two-variable problem:
\begin{equation*}
\sup_{\substack{\theta_1,\theta_2>0\\ \theta_1>\theta_2} }F_{\nu}(\theta_1,\theta_2,\theta_2)
\end{equation*}
 by observing a hidden relation of this problem with the following elementary question (See Lemma \ref{SPlemma}.):
\begin{center}
``When does $xy+yz+zx$ attains its minimum with $x+y+z$ and $xyz$ fixed$\,?$"
\end{center}
 Finally, the following simple scaling property of the 
functional \eqref{we re2}:
\[
	F_{\nu}(\theta_1,\theta_2,\theta_3)=F_{\nu}(k\theta_1,k\theta_2,k\theta_3),\quad\forall~k>0,
	\] 
together with a proper choice of $k$ enable one to  reduce the optimization problem further into a much manageable problem involving only one variable given in \eqref{cnux}, from which
various useful information on the entropy-entropy production estimate  can be derived.

\subsection{Literature} The ES-BGK model was suggested by Halway \cite{Holway} to remedy the incorrect production of transport coefficients of the original BGK model \cite{BGK}.
It was, however, somewhat forgotten in the literature since it was not known at the time whether the H-theorem holds for the ES-BGK model.
The proof of the H-theorem was provided in \cite{ALPP}, due to which the model got popularized in the kinetic community \cite{AAS,ABLP,CXC comparison,GT, KA,MWRZ,MZHRS,MS}. A systematic derivation of the model is considered in \cite{Brull Jac ES}. ES-BGK model for gas mixture is suggested in \cite{Brull ES poly,GMS ES,MK}.  Several numerical methods for the ES-BGK model were developed in \cite{ABLP,FJ,Russo Yun}. 

The existence theory was considered in various contexts: near-equilibrium regime \cite{Li Sun,Yun SIAM},
weak solutions subject to general initial data \cite{P Yun JMP,Zhang Hu}, mild solutions with strong decay in velocity  \cite{Chen,P Yun JMP},  
and stationary flows \cite{Bang Yun,Brull Yun}. The entropy-entropy production estimate was derived in \cite{Yun Entropy} for $-1/2<\nu<1$. Entropic properties below the correct Prandtl number $3/2$ are considered in \cite{THM}. For relevant results on polyatomic generalizations of the ES-BGK model, we refer to \cite{BCRY,KA,KKA,P Yun JDE,P Yun AML,Yun Dichotomy}.\newline

The paper is organized as follows. In Section 2, we prove the main Theorem assuming the validity of Lemma \ref{main lemma}. Then the proof of Lemma \ref{main lemma} is given in Section 3. In Section 4, the point on which the $C_{\nu}$ is attained is characterized.

%
%
%
%
\section{Proof of the main theorem}
\subsection{Proof of Theorem \ref{main result} (1)}
	For simplicity, we denote
	\[
	\widetilde{C}_{\nu}=\frac{1}{3}\nu^2(5-2\nu).
	\]
	Then, it suffices to prove that
	\begin{align}\label{inequal}
	\frac{\displaystyle3\ln \left(1+\frac{1}{3}x\right)- \ln \left(1+\frac{1+2\nu}{3}x \right)-2\ln \left(1+\frac{1-\nu}{3}x \right)}{\displaystyle3\ln \left(1+\frac{1}{3}x\right)-  \ln(1+x)}\leq \widetilde C_\nu\quad(\forall~x>0),
	\end{align}
	or equivalently,
	\begin{equation}\label{equiv}
	\left(1+\frac{1+2\nu}{3}x \right)\left(1+\frac{1-\nu}{3}x \right)^2\geq (1+x)^{\widetilde C_\nu}\left(1+\frac{1}{3}x\right)^{3(1-\widetilde C_\nu)}\quad(\forall~x>0).
	\end{equation}

	For this, we compute
	\begin{align}\label{from below}
	\begin{split}
	&\left(1+\frac{1+2\nu}{3}x \right)\left(1+\frac{1-\nu}{3}x \right)^2\\
	&\qquad=1+ x+\frac{1-\nu^2}{3}x^2+\left(\frac{1+2\nu}{3} \right)\left(\frac{1-\nu}{3} \right)^2x^3\\
	&\qquad=(3\nu^2-2\nu^3)+ (3\nu^2-2\nu^3)x+\frac{1}{3}(2\nu^2-2\nu^3)x^2 +(1+2\nu)(1-\nu)^2\left(1+\frac{1}{3}x\right)^3\\
	&\qquad=\frac{1}{3}\nu^2\left((9 -6\nu)(1+x)+ (2 -2\nu)x^2\right)  +(1+2\nu)(1-\nu)^2\left(1+\frac{1}{3}x\right)^3\\
	&\qquad=\frac{1}{3}\nu^2 (1 +2\nu)(1+x)+ \frac{8}{3}\nu^2 (1-\nu)\left(1+x+\frac{1}{4}x^2\right)  +(1+2\nu)(1-\nu)^2\left(1+\frac{1}{3}x\right)^3.
	\end{split}
	\end{align}
	We then observe
	\[
	\left(1+x+\frac{1}{4}x^2\right)^2-(1+x) \left(1+\frac{1}{3}x\right)^3=\frac{1}{6}x^2+\frac{7}{54}x^3+\frac{11}{432}x^4>0
	\]
	to bound the last line of \eqref{from below} from below by
	\begin{align*}
	\frac{1}{3}\nu^2 (1 +2\nu)(1+x)+ \frac{8}{3}\nu^2 (1-\nu)(1+x)^\frac{1}{2}\left(1+\frac{1}{3}x\right)^\frac{3}{2}  +(1+2\nu)(1-\nu)^2\left(1+\frac{1}{3}x\right)^3,
	\end{align*}
	which, with the help of the weighted AM-GM inequality, can be bounded further from below by
	\begin{align*}
	&(1+x)^{\frac{1}{3}\nu^2 (1 +2\nu)} \left[(1+x)^\frac{1}{2}\left(1+\frac{1}{3}x\right)^\frac{3}{2}\right]^{\frac{8}{3}\nu^2 (1-\nu)}  \left[\left(1+\frac{1}{3}x\right)^3\right]^{(1+2\nu)(1-\nu)^2}\\
	&\qquad=(1+x)^{\frac{1}{3}\nu^2 (5-2\nu)}  \left(1+\frac{1}{3}x\right)^{3-5\nu^2+2\nu^3}\cr
	&\qquad= (1+x)^{\widetilde C_\nu}\left(1+\frac{1}{3}x\right)^{3(1-\widetilde C_\nu)},
	\end{align*}
	which gives (\ref{equiv}).
\begin{remark}
	We mention that the equality holds when $\nu=0$, since both sides of the inequality (\ref{inequal}) reduces to 0 at $\nu=0$.
	In Section 4, we will prove that the equality holds also in the case $\nu=-1/2$.
\end{remark}

\subsection{Proof of Theorem \ref{main result} (2)}
We now move on to the proof of the entropy-entropy production estimate. The following lemma that compares the differences of H-functionals of various Maxwellians, is the key element 
in the proof.
\begin{lemma}\label{main lemma}
	With the $C_{\nu}$ defined in Theorem \ref{main result}, we have
	\begin{align*}
	H(\mathcal{M}_\nu)-H(\mathcal{M}_0)\leq C_\nu\{H(\mathcal{M}_1)-H(\mathcal{M}_0)\},
	\end{align*}
	for $-1/2\leq\nu<1$.
\end{lemma}
\begin{remark}
This improves the corresponding estimate in \cite{Yun Entropy} in that (1)  the critical case $\nu=-1/2$ is covered, and (2) $C_{\nu}$ is sharper than the one obtained in \cite{Yun Entropy}.
\end{remark}

Assuming the validity of the Lemma \ref{main lemma}, we can prove the main theorem as follows.
First we use the convexity of $x\ln x$ and $\frac{d}{dx}(x\ln x)=1+\ln x$ to get 
		\begin{align*}
D_{\nu}(f)=\int A_\nu (1+\ln f)(\mathcal{M}_{\nu}-f)
\leq A_\nu\{H(\mathcal{M}_{\nu})-H(f)\}.
\end{align*}
Thanks to Lemma \ref{main lemma}, the R.H.S. can be treated as 
\begin{align}\label{last line}
\begin{split}
H\big(\mathcal{M}_{\nu}\big)-H(f)
&=-\{H(f)-H(\mathcal{M}_{0})\}+\{H(\mathcal{M}_{\nu})-H(\mathcal{M}_{0}))\}\cr
&\leq -H(f|\mathcal{M}_{0})+
C_{\nu}\left\{H(\mathcal{M}_{1})-H(\mathcal{M}_{0})\right\}.
\end{split}
\end{align}

Then we recall the following inequality from \cite{ALPP,Yun Entropy}
\[
H(\mathcal{M}_{0})\leq H(\mathcal{M}_{1})\leq H(f).
\]
to bound the last line of (\ref{last line}) by
\begin{align*}
& -H(f|\mathcal{M}_{0})+C_{\nu}\left\{H(f)-H(\mathcal{M}_{0})\right\}=-(1-C_{\nu})H(f|\mathcal{M}_{0}),
\end{align*} 
which completes the proof.
The proof of Lemma \ref{main lemma} is considered in the next section.
%
%
%
%
%
\section{Proof of Lemma \ref{main lemma}}
The proof of Lemma \ref{main lemma} is obtained by first formulating it as a three-variable  optimization problem and  then reducing it into a one-variable problem. As such, we divide this  section into several steps.
\subsection{Step 1: Reformulation as a maximization problem}
Let  $\theta_i$ ($i=1,2,3$) be three eigenfunctions of the symmetric matrix $\Theta$.\newline

\noindent$\bullet$ Case i) $\theta_1=\theta_2=\theta_3$:
In this case, $\theta_i=T$, so that the temperature tensor $\mathcal{T}_{\nu}$ reduces to $TI$ and $\mathcal{M}_{\nu}$ reduces to $\mathcal{M}_0$, from which the desired inequality follows trivially.\newline

\noindent$\bullet$ Case ii) At least one of $\theta_i$ takes a different value: The assumption $\rho>0$ implies that 
\[
\kappa^{\top}\Theta\kappa=\int_{\mathbb{R}^3}f(x,v,t)\{(v-U)\cdot\kappa\}^2dv>0\qquad \kappa\in\mathbb{R}^3.
\]
This gives $\theta_i>0$.
First, we recall from  Lemma 2.2 in \cite{Yun Entropy} that
\begin{align*}
H(\mathcal{M}_0)-H(\mathcal{M}_\nu)&=\frac{\rho}{2}\ln{\frac{\prod_{i=1}^{3}\{(1-\nu)\frac{\theta_1+\theta_2+\theta_3}{3}+\nu\theta_i\}}{(\frac{\theta_1+\theta_2+\theta_3}{3})^3}},\\
H(\mathcal{M}_0)-H(\mathcal{M}_1)&=\frac{\rho}{2}\ln{\frac{\theta_1\theta_2\theta_3}{(\frac{\theta_1+\theta_2+\theta_3}{3})^3}}.
\end{align*}

Therefore, the inequality in Lemma \ref{main lemma}:
\begin{align*}
 H(\mathcal{M}_0)-H(\mathcal{M}_\nu)\geq C_\nu\{H(\mathcal{M}_0)-H(\mathcal{M}_1)\},
\end{align*}
can be rewriten as
\begin{equation}\label{B-1}
 \ln{\frac{\prod_{i=1}^{3}\{(1-\nu)\frac{\theta_1+\theta_2+\theta_3}{3}+\nu\theta_i\}}{(\frac{\theta_1+\theta_2+\theta_3}{3})^3}}\geq  C_\nu\ln{\frac{\theta_1\theta_2\theta_3}{(\frac{\theta_1+\theta_2+\theta_3}{3})^3}}.
\end{equation}
Here we divided both sides by $\rho$ using the assumption $\rho\neq0$.
From our assumption on $\theta_i$ that at least one pair of $\theta_i$ takes a different value, and the
fact that the equality in the Inequality of arithmetic and geometric means holds only when all the variables take the same value, the R.H.S. of (\ref{B-1}) is strictly less than zero.
Therefore (\ref{B-1}) is equivalent to 
\[
\ln{\frac{\prod_{i=1}^{3}\{(1-\nu)\frac{\theta_1+\theta_2+\theta_3}{3}+\nu\theta_i\}}{(\frac{\theta_1+\theta_2+\theta_3}{3})^3}}\bigg/  \ln{\frac{\theta_1\theta_2\theta_3}{(\frac{\theta_1+\theta_2+\theta_3}{3})^3}} \leq C_\nu.
\]
From this, we conclude that, to derive the desired inequality, it is enough to study the range of the functional $F_\nu$ defined by
	\begin{align*}
\begin{split}
F_{\nu}(\theta_1,\theta_2,\theta_2)=  \frac{3\ln \left(\frac{\theta_1+\theta_2+\theta_3}{3}\right)-\ln \left\{\prod\limits_{i=1}^3\big((1-\nu)\frac{\theta_1+\theta_2+\theta_3}{3}+\nu\theta_i\big)\right\} }{3\ln \left(\frac{\theta_1+\theta_2+\theta_3}{3}\right)- \ln(\theta_1\theta_2\theta_3)}.
\end{split}
\end{align*}
More precisely, it is sufficient to prove that the following identity holds:
\begin{equation}\label{cdef}
C_\nu=\sup_{\substack{\theta_1,\theta_2,\theta_3>0\\ \exists~i,j:~\theta_i\neq\theta_j}}F_{\nu}(\theta_1,\theta_2,\theta_3)
\end{equation}
for each fixed $\nu\in[-1/2,1)$, where $C_{\nu}$ is the coefficient defined in (\ref{cnux}). The rest of this section is devoted to proof of (\ref{cdef}).
\subsection{Step II: Reduction into a two-variable problem} 
We first show that the optimization problem (\ref{cdef})  can be reduced into an optimization problem of two-variable problem. We need the following lemma on the range of $xy+yz+zx$ when the value of $x+y+z$ and $xyz$ are fixed.
\begin{lemma}\cite{Barbara}\label{SPlemma}
	Fix $P,S>0$ such that $P<S^3/27$. 
	Let $x,y,z>0$ vary satisfying the relation:
	\[
	x+y+z=S ~\mbox{ and }~ xyz=P,
	\]  
	Set $k=4P/S^3$ $($so, $0<k<4/27$$)$. 
	Then, we have
	\begin{enumerate}
		\item The range of $xy+yx+zs$ is 
		\[
		\frac{1}{4}S^2(4\alpha-3\alpha^2)\leq xy+yz+zx\leq \frac{1}{4}S^2(4\beta-3\beta^2),
		\]
		where  $\alpha$ and $\beta$ are solutions of $x^2-x^3=k$ in the range $(0,2/3)$ and $(2/3,1)$
		respectively.
		\item The lower bound  is attained if and only if $x,y,z$ are 
		\[
		\frac{1}{2}S\alpha, ~\frac{1}{2}S\alpha, ~S(1-\alpha)
		\]
		in some order, and the upper bound is attained if and only if $x,y,z$ are
		\[
		\frac{1}{2}S\beta,~ \frac{1}{2}S\beta,~ S(1-\beta)
		\] 
		in some order.
	\end{enumerate}
\end{lemma}
\begin{proof}
	See \cite{Barbara}.
\end{proof}
\begin{remark}
	In Lemma \ref{SPlemma}, we have
	\[
	\frac{1}{2}S\alpha<S(1-\alpha)\quad\mbox{and}\quad \frac{1}{2}S\beta>S(1-\beta).
	\]
\end{remark}
We are now ready to derive our first reduction. 
\begin{lemma}\label{lemma 3.3} For $\nu\in[-1/2,1)$, $F_{\nu}$ satisfies
	\begin{equation*}
	\sup_{\substack{\theta_1,\theta_2,\theta_3>0\\ \exists~i,j:~\theta_i\neq\theta_j}}F_{\nu}(\theta_1,\theta_2,\theta_3)=\sup_{\substack{\theta_1,\theta_2,\theta_3>0\\ \theta_1>\theta_2=\theta_3} }F_{\nu}(\theta_1,\theta_2,\theta_3).
	\end{equation*}
	\end{lemma}
\begin{proof}
 From the obvious relation
	\[
	\big\{(\theta_1,\theta_2,\theta_3)\in \mathbb R_+^3:\exists~i,j:\theta_i\neq\theta_j\big\}\supset\big\{(\theta_1,\theta_2,\theta_3)\in\mathbb R_+^3:\theta_1>\theta_2=\theta_3\big\},
	\]
	we have
	\begin{equation}\label{we have}
	\sup_{\substack{\theta_1,\theta_2,\theta_3>0\\ \exists~i,j:~\theta_i\neq\theta_j}}F_{\nu}(\theta_1,\theta_2,\theta_3)\geq\sup_{\substack{\theta_1,\theta_2,\theta_3>0\\ \theta_1>\theta_2=\theta_3} }F_{\nu}(\theta_1,\theta_2,\theta_3).
	\end{equation}
	Therefore, it suffices to prove the reverse inequality. For this, we fix $\theta_1,\theta_2,\theta_3>0$ satisfying $\theta_i\neq\theta_j$ for some $i,j$ and set
	\[
	S:=\theta_1+\theta_2+\theta_3,\quad \Delta:=\theta_1\theta_2+\theta_2\theta_3+\theta_3\theta_1 \quad\mbox{and}\quad  P:=\theta_1\theta_2\theta_3.
	\]
	By Lemma \ref{SPlemma}, we can choose $\tilde\theta_1,\tilde\theta_2,\tilde\theta_3>0$ such that $\tilde\theta_1>\tilde\theta_2=\tilde\theta_3$ and
	\[
	S=\tilde\theta_1+\tilde\theta_2+\tilde\theta_3,\quad \Delta\geq \tilde\theta_1\tilde\theta_2+\tilde\theta_2\tilde\theta_3+\tilde\theta_3\tilde\theta_1 \quad\mbox{and}\quad  P=\tilde\theta_1\tilde\theta_2\tilde\theta_3.
	\]
For such choice of $\tilde{\theta}_i$, we observe
	\begin{align*}
	\prod\limits_{i=1}^3\left\{(1-\nu)\bigg(\frac{S}{3}\bigg)+\nu\theta_i\right\}&=(1-\nu)^3\bigg(\frac{S}{3}\bigg)^3+\nu(1-\nu)^2(\theta_1+\theta_2+\theta_3)\bigg(\frac{S}{3}\bigg)^2\\
	&\quad+\nu^2(1-\nu)(\theta_1\theta_2+\theta_2\theta_3+\theta_3\theta_1)\bigg(\frac{S}{3}\bigg)+\nu^3P\\
	&\geq(1-\nu)^3\bigg(\frac{S}{3}\bigg)^3+\nu(1-\nu)^2(\tilde\theta_1+\tilde\theta_2+\tilde\theta_3)\bigg(\frac{S}{3}\bigg)^2\\
	&\quad+\nu^2(1-\nu)(\tilde\theta_1\tilde\theta_2+\tilde\theta_2\tilde\theta_3+\tilde\theta_3\tilde\theta_1)\bigg(\frac{S}{3}\bigg)+\nu^3P\\
	&=\prod\limits_{i=1}^3\left\{(1-\nu)\bigg(\frac{S}{3}\bigg)+\nu\tilde\theta_i\right\}.	
	\end{align*}
Hence we have	
	\begin{align*}
F_{\nu}(\theta_1,\theta_2,\theta_3)&=\frac{3\ln \frac{S}{3}-\ln \left[\prod\limits_{i=1}^3\left\{(1-\nu)\frac{S}{3}+\nu\theta_i\right\}\right] }{3\ln \frac{S}{3}- \ln P}\\
&\leq \frac{3\ln \frac{S}{3}-\ln \left[\prod\limits_{i=1}^3\left\{(1-\nu)\frac{S}{3}+\nu\tilde\theta_i\right\}\right] }{3\ln \frac{S}{3}- \ln P}\\
&=F_{\nu}(\tilde\theta_1,\tilde\theta_2,\tilde\theta_3),
\end{align*}
which gives
	\begin{equation}\label{we have2}
	\sup_{\substack{\theta_1,\theta_2,\theta_3>0\\ \exists~i,j:~\theta_i\neq\theta_j}}F_{\nu}(\theta_1,\theta_2,\theta_3)\leq\sup_{\substack{\theta_1,\theta_2,\theta_3>0\\ \theta_1>\theta_2=\theta_3} }F_{\nu}(\theta_1,\theta_2,\theta_3).
	\end{equation}
	Combination of (\ref{we have}) and (\ref{we have2}) establishes the desired result.
\end{proof}

\subsection{Step III: Reduction into a one-variable problem} 
In the following lemma, we reduce the problem further into a one-variable optimization problem to complete
the proof of Lemma \ref{main lemma}.
\begin{lemma}\label{prop1}
	Let $\nu\in[-1/2,1)$, then we have
	\begin{align*}
	\sup_{\substack{\theta_1,\theta_2,\theta_3>0\\ \exists~i,j:~\theta_i\neq\theta_j}}F_{\nu}(\theta_1,\theta_2,\theta_3)
	&=C_{\nu}.
	\end{align*}
	for $C_{\nu}$ defined in (\ref{cnux}) as the maximum value of an one-variable optimization problem. 
\end{lemma}
\begin{proof}
	By Lemma \ref{lemma 3.3}, it is enough to consider
	
	\begin{align}
	\begin{split}
	F_{\nu}(\theta_1,\theta_2,\theta_2)=  \frac{3\ln \left(\frac{\theta_1+2\theta_2}{3}\right)-\ln \left\{(1-\nu)\frac{\theta_1+\theta_2+\theta_3}{3}+\nu\theta_1\right\}
	\left\{(1-\nu)\frac{\theta_1+2\theta_2}{3}+\nu\theta_2\right\}^2 }{3\ln \left(\frac{\theta_1+2\theta_2}{3}\right)- \ln(\theta_1\theta^2_2)}.
	\end{split}
	\end{align}
    for $\theta_1>\theta_2$ without loss of generality.
	
If we denote by $F_{\nu,1}$, $F_{\nu,2}$ the numerator and the denominator of $F_{\nu}$ respectively, we observe the following scaling property:
	\[
	F_{\nu,i}(\theta_1,\theta_2,\theta_3)=F_{\nu,i}(k\theta_1,k\theta_2,k\theta_3),\quad\forall~k>0.
	\]
	Therefore, taking $k=\theta^{-1}_2$, we get
		
	\begin{align}
	\begin{split}
	&F_{\nu,1}(\theta_1,\theta_2,\theta_2)\\
	&\qquad=  3\ln \left(\frac{\theta_1+2\theta_2}{3}\right)-\ln \left[(1-\nu)\frac{\theta_1+2\theta_2}{3}+\nu\theta_1\right] 
	\left[(1-\nu)\frac{\theta_1+2\theta_2}{3}+\nu\theta_2\right]^2\cr  
		&\qquad=  3\ln \left(\frac{\frac{\theta_1}{\theta_2}+2}{3}\right)-\ln \left[(1-\nu)\frac{\frac{\theta_1}{\theta_2}+2}{3}+\nu\frac{\theta_1}{\theta_2}\right]
	\left[(1-\nu)\frac{\frac{\theta_1}{\theta_2}+2}{3}+\nu\right]^2\cr
	 	&\qquad=  3\ln \left[1+\frac{1}{3}\left(\frac{\theta_1}{\theta_2}-1\right)\right]-\ln \left[1+\frac{1+2\nu}{2}\left(\frac{\theta_1}{\theta_2}-1\right)\right]
	 \left[1+\frac{1-\nu}{3}\left(\frac{\theta_1}{\theta_2}-1\right)\right]^2.
	\end{split}
	\end{align}
	We then set 
	\[
	x=\frac{\theta_1}{\theta_2}-1
	\]
	to obtain
		\begin{align}
	F_{\nu,1}(\theta_1,\theta_2,\theta_2)=  3\ln \left(1+\frac{1}{3}x\right)-\ln \left(1+\frac{1+2\nu}{2}x\right)
	-2\ln\left(1+\frac{1-\nu}{3}x\right).
	\end{align}
Similarly, we have
		
		\begin{align}
		F_{\nu,2}(\theta_1,\theta_2,\theta_2)=3\ln \left(1+\frac{1}{3}x\right)- \ln(1+x).
		\end{align}
This gives the desired result.
\end{proof}

%
%
%
%

\section{Further consideration of $C_{\nu}$}
In this section, we characterize the value  at which the optimal value of $C_{\nu}$ is attained. 
For this, we define
\[
G_\nu(x):=\frac{\displaystyle3\ln \left(1+\frac{1}{3}x\right)- \ln \left(1+\frac{1+2\nu}{3}x \right)-2\ln \left(1+\frac{1-\nu}{3}x \right)}{\displaystyle3\ln \left(1+\frac{1}{3}x\right)-  \ln(1+x)}
\]
and let $G_{\nu,1}$ and $G_{\nu,2}$ denote the numerator and the denominator of $G_{\nu}$ respectively.
\begin{proposition}\label{prop2}
	For $\nu\in[-1/2,1)$, let $C_\nu$ be given by the relation \eqref{cdef}. Then we have\newline
	\noindent (1) $C_{\nu}=0$ if $\nu=0.$\newline
	\noindent (2) $C_{\nu}=1/2$ if $\nu=-1/2$\newline
	\noindent (3) If $\nu\in (-1/2,0)\cup(0,1)$, then $C_\nu$ is characterized by 
	  \[C_{\nu}= \frac{G_{\nu,1}(x_\nu)}{G_{\nu,2}(x_\nu)},\] 
	  where 
	  $0<x_\nu<\infty$ is the unique solution to
	  \[ 
	  \frac{G_{\nu,1}(x_\nu)}{G_{\nu,2}(x_\nu)}=\frac{G'_{\nu,1}(x_\nu)}{G'_{\nu,2}(x_\nu)}.
	  \]
\end{proposition}
\begin{proof}
	\noindent (1) When $\nu=0$,  $G_{\nu,1}(x)\equiv0$, Therefore,  by Lemma \ref{prop1}, we have $C_\nu=0$.\\\\
\noindent (2) The case $\nu=-\frac{1}{2}$: 
	Note that
\begin{align*}
G'_{\nu,1}(x)&=  \frac{1}{1+\frac{1}{3}x}- \frac{1+2\nu}{3}\frac{1}{1+\frac{1+2\nu}{3}x}-2\frac{1-\nu}{3}\frac{1}{1+\frac{1-\nu}{3}x }\\
&=\frac{\frac{2}{3}\nu^2x}{(1+\frac{1}{3}x)(1+\frac{1+2\nu}{3}x)(1+\frac{1-\nu}{3}x)},\\
G'_{\nu,2}(x)&=\frac{1}{1+\frac{1}{3}x}-  \frac{1}{1+x}=\frac{\frac{2}{3}x}{(1+\frac{1}{3}x)(1+x)},
\end{align*}
so that
\begin{align*}
\frac{G'_{\nu,1}(x)}{G'_{\nu,2}(x)}&=\frac{ \nu^2(1+x)}{\left(1+\frac{1+2\nu}{3}x\right)\left(1+\frac{1-\nu}{3}x\right)},
\end{align*}
and
\begin{align}\label{and}
\frac{d}{dx}\left( \frac{G'_{\nu,1}(x)}{G'_{\nu,2}(x)} \right)&=\frac{ \nu^2(\frac{1-\nu}{3})\left\{\frac{4+2\nu}{3}-(\frac{1+2\nu}{3})(1+x)^2\right\}}{\left(1+\frac{1+2\nu}{3}x\right)^2
	\left(1+\frac{1-\nu}{3}x\right)^2}.
\end{align}

Therefore, when $\nu=-1/2$, we have 
	\begin{align*}
	\frac{d}{dx}&\left( \frac{G'_{\nu,1}(x)}{G'_{\nu,2}(x)} \right)=\frac{1}{8(1+\frac{1}{2}x)^2}>0.
	\end{align*}
and
	\begin{align*}
	\frac{d}{dx}\left(G_{\nu,2}(x) \frac{G'_{\nu,1}(x)}{G'_{\nu,2}(x)}-G_{\nu,1}(x)  \right)=G_{\nu,2}(x)\frac{d}{dx}\left( \frac{G'_{\nu,1}(x)}{G'_{\nu,2}(x)} \right)>0,\quad\forall~x>0.
	\end{align*}
	This implies that $G_{\nu,2}(x) \frac{G'_{\nu,1}(x)}{G'_{\nu,2}(x)}-G_{\nu,1}(x)$ is a strictly increasing function, so that
	\begin{align*}
	G_{\nu,2}(x) \frac{G'_{\nu,1}(x)}{G'_{\nu,2}(x)}-G_{\nu,1}(x) >\lim_{x\searrow0}\left( G_{\nu,2}(x) \frac{G'_{\nu,1}(x)}{G'_{\nu,2}(x)}-G_{\nu,1}(x) \right)=0\cdot \frac{1}{4}-0=0,\quad\forall~x>0.
	\end{align*}
	Hence we have 
	\begin{equation*}
	\begin{split}
	\frac{d}{dx}\left( \frac{G_{\nu,1}(x)}{G_{\nu,2}(x)} \right)=\frac{G'_{\nu,2}(x)}{(G_{\nu,2}(x))^2} \left(G_{\nu,2}(x)\frac{G'_{\nu,1}(x)}{G'_{\nu,2}(x)}-G_{\nu,1}(x)  \right)>0,\quad\forall~x>0.
	\end{split}
	\end{equation*}
	This shows that $G_{\nu,1}(x)/G_{\nu,2}(x)$ also is a strictly increasing function. Therefore, we have
	\begin{align*}
	C_{-\frac{1}{2}}&=\lim_{x\rightarrow\infty}\frac{G_{\nu,1}(x)}{G_{\nu,2}(x)},
	\end{align*}
	and the use of L'Hospital's rule gives the desired result:
	\begin{align*}
	C_{-\frac{1}{2}}&=\lim_{x\to\infty}\frac{G'_{\nu,1}(x)}{G'_{\nu,2}(x)}\cr
	&=\lim_{x\to\infty}\frac{1/\left(1+\frac{1}{3}x\right)-  1/\left(1+\frac{1}{2}x \right)}{1/\left(1+\frac{1}{3}x\right)-  1/(1+x)}\\
	&=\frac{1}{2}.
	\end{align*}
	\noindent(3) $ \nu\in\left(-\frac{1}{2},0\right)\cup(0,1)$: Set $x_*:=\sqrt{\frac{4+2\nu}{1+2\nu}}-1$. Then we observe from (\ref{and}) that 
	\[
	\frac{d}{dx}\left(G_{\nu,2}(x) \frac{G'_{\nu,1}(x)}{G'_{\nu,2}(x)}-G_{\nu,1}(x)  \right)=G_{\nu,2}(x)\frac{d}{dx}\left( \frac{G'_{\nu,1}(x)}{G'_{\nu,2}(x)} \right)\begin{cases}
	\displaystyle >0\quad\mbox{if}\quad 0<x<x_*,\\
	\displaystyle <0\quad\mbox{if}\quad x_*<x<\infty.
	\end{cases}
	\]
	Hence
	\[  G_{\nu,2}(x) \frac{G'_{\nu,1}(x)}{G'_{\nu,2}(x)}-G_{\nu,1}(x) >\lim_{x\searrow0}\left( G_{\nu,2}(x) \frac{G'_{\nu,1}(x)}{G'_{\nu,2}(x)}-G_{\nu,1}(x) \right)=0\cdot \nu^2-0=0,\quad\forall~x\in(0,x_*].
	\]
	Note that
	\begin{align*}
	\lim_{x\to\infty}\left(G_{\nu,2}(x) \frac{G'_{\nu,1}(x)}{G'_{\nu,2}(x)}-G_{\nu,1}(x)\right) &=0-\left[3\ln \left(\frac{1}{3}\right)- \ln \left(\frac{1+2\nu}{3} \right)-2\ln \left(\frac{1-\nu}{3} \right)\right]\\
	&=\ln\{1-\nu^2(3-2\nu)\}<0.
	\end{align*}
	By the intermediate value theorem, there exists  $x_\nu\in(x_*,\infty)$ satisfying
	\[
	G_{\nu,2}(x) \frac{G'_{\nu,1}(x)}{G'_{\nu,2}(x)}-G_{\nu,1}(x)  \begin{cases}
	\displaystyle >0\quad\mbox{if}\quad 0<x<x_\nu,\\
	\displaystyle =0\quad\mbox{if}\quad x=x_\nu,\\
	\displaystyle <0\quad\mbox{if}\quad x_\nu<x<\infty.
	\end{cases}
	\]
	Therefore,
	\[
	\frac{d}{dx}\left( \frac{G_{\nu,1}(x)}{G_{\nu,2}(x)} \right)=\frac{G'_{\nu,2}(x)}{(G_{\nu,2}(x))^2} \left(G_{\nu,2}(x)\frac{G'_{\nu,1}(x)}{G'_{\nu,2}(x)}-G_{\nu,1}(x)  \right)  \begin{cases}
	\displaystyle >0\quad\mbox{if}\quad 0<x<x_\nu,\\
	\displaystyle =0\quad\mbox{if}\quad x=x_\nu,\\
	\displaystyle <0\quad\mbox{if}\quad x_\nu<x<\infty,
	\end{cases}
	\]
	from which we find
	\[
	C_{\nu}=\sup_{0<x<\infty}\frac{G_{\nu,1}(x)}{G_{\nu,2}(x)}= \frac{G_{\nu,1}(x_\nu)}{G_{\nu,2}(x_\nu)}.
	\]
\end{proof}


\noindent{\bf Acknowledgement:}
Doheon Kim was supported by a KIAS Individual Grant (MG073901) at Korea Institute for Advanced Study.
The research of Seok-Bae Yun was supported by Samsung Science and Technology Foundation under Project Number SSTF-BA1801-02.

\end{document}